\documentclass[10pt,a4paper]{amsart}

\theoremstyle{plain}
\newtheorem{theorem}{Theorem}[section]
\newtheorem{lemma}[theorem]{Lemma}
\newtheorem{proposition}[theorem]{Proposition}
\newtheorem{corollary}[theorem]{Corollary}

\theoremstyle{definition}

\theoremstyle{remark}
\newtheorem{remark}[theorem]{Remark}

\newtheorem{notations}[theorem]{Notations}

\newcommand{\Ps}{\mathbb{P}}

\def\bin #1#2 {\left( \matrix { #1 \cr #2 \cr } \right) }

\begin{document}

\title[Refining Castelnuovo-Halphen  bounds]
{Refining  Castelnuovo-Halphen  bounds}

\author{Vincenzo Di Gennaro }
\address{Universit\`a di Roma \lq\lq Tor Vergata\rq\rq, Dipartimento di Matematica,
Via della Ricerca Scientifica, 00133 Roma, Italy.}
\email{digennar@axp.mat.uniroma2.it}

\author{Davide Franco }
\address{Universit\`a di Napoli
\lq\lq Federico II\rq\rq, Dipartimento di Matematica e
Applicazioni \lq\lq R. Caccioppoli\rq\rq, P.le Tecchio 80, 80125
Napoli, Italy.} \email{davide.franco@unina.it}

\bigskip
\abstract Fix integers $r,d,s,\pi$ with $r\geq 4$, $d\gg s$,
$r-1\leq s \leq 2r-4$, and $\pi\geq 0$. Refining classical results
for the genus of a projective curve, we exhibit a sharp upper
bound for the arithmetic genus $p_a(C)$ of an integral projective
curve $C\subset {\mathbb{P}^r}$ of degree $d$,  assuming that $C$
is not contained in any surface of degree $<s$, and not contained
in any surface of degree $s$ with sectional genus $> \pi$. Next we
discuss other types of bound for $p_a(C)$, involving conditions on
the entire Hilbert polynomial of the integral surfaces on which
$C$ may lie.



\endabstract

\maketitle

\section{Introduction}

A classical problem in the theory of projective curves is the
classification of all their possible genera in terms of the degree
$d$ and the dimension $r$ of the space where they are embedded. In
1882  Halphen \cite{Halphen} and Noether \cite{Noether} determined
an upper bound $G(3,d)$ for the genus of an irreducible, non
degenerate curve in $\Ps^3$, and in 1889 Castelnuovo
\cite{Castelnuovo} found the analogous bound $G(r,d)$ for the
genus of irreducible, non degenerate curves in $\Ps^r$, $r\geq 3$.

Since curves of maximal genus $G(3,d)$ in $\Ps^3$ must lie on a
quadric surface, it is natural to ask for the maximal genus
$G(3,d,s)$ of space curves of degree $d$, not contained in
surfaces of degree less than a fixed integer $s$. In fact Halphen
gave such a refined bound. His argument was not complete, but in
1977 Gruson and Peskine \cite{GP} provided a complete proof in the
range $d>s^2-s$.

The same phenomenon occurs for curves of maximal genus $G(r,d)$ in
$\Ps^r$, also called {\it Castelnuovo's curves}: at least when
$d>2r$, they must lie on surfaces of minimal degree $r-1$. As
before, one may refine Castelnuovo's bound, looking for the
maximal genus $G(r,d,s)$ of curves of degree $d$ in $\Ps^r$, not
contained in surfaces of degree less than a fixed integer $s$. In
1982 Eisenbud and Harris (\cite{EH}, Theorem (3.22), p. 117)
determined such a bound for $r-1\leq s \leq 2r-2$ and $d\gg s$.
Next, in 1993, the bound $G(r,d,s)$ has been computed for any $s$
and   $d\gg s$ (see \cite{CCD}).

A very special feature of the curves of maximal genus $G(r,d,s)$,
which generalizes what we said about Castelnuovo's curves (i.e.
when $s=r-1$), is that they must lie on {\it Castelnuovo's
surfaces} of degree $s$, i.e. on surfaces whose general hyperplane
sections are themselves curves of maximal genus $G(s,r-1)$ in
$\Ps^{r-1}$ (see \cite{CCD}).

Therefore, pushing further previous analysis, one may ask for {\it
the maximal genus $G(r,d,s,\pi)$ of curves of degree $d$, not
contained in surfaces of degree $<s$, {\it neither} in surfaces of
degree $s$ with sectional genus greater than a fixed integer
$\pi$} (e.g. $\pi=G(r-1,s)-1$). Of course, one may assume  $0\leq
\pi\leq G(r-1,s)$, and for $\pi=G(r-1,s)$ and $d\gg s$ we have
$G(r,d,s,\pi)=G(r,d,s)$.

In the present paper we compute $G(r,d,s,\pi)$, in the range
$r-1\leq s\leq 2r-4$ and $d\gg s$ (Theorem \ref{first})(except for
the cases $s=2r-3$ and $s= 2r-2$, it is the quoted Eisenbud-Harris
range for $s$ \cite{EH}). Next we discuss other types of bound for
$p_a(C)$, involving conditions on the entire Hilbert polynomial of
the integral surfaces on which $C$ may lie (Proposition
\ref{second}).

\bigskip
\section{Notations and the statement of the main results}
In order to state our results we need some preliminary notation,
which we will use throughout the paper.

\begin{notations}\label{notazioni}
(i) Fix integers $r$, $d$, $s$, $\pi$ and $p$, with $r\geq 3$ and
$s\geq r-1$. Define $m$ and $\epsilon$ by dividing
$d-1=ms+\epsilon$, $0\leq \epsilon \leq s-1$. Set $
\pi_0:=\pi_0(s,r-1):=s-r+1$. Notice that when $r-1\leq s\leq 2r-3$
then $\pi_0=G(r-1,s)$, i.e. $\pi_0$ is the Castelnuovo's bound for
a curve of degree $s$ in $\Ps^{r-1}$ \cite{EH}. Set
$$
d_0(r):=
\begin{cases}
16(r-2)(2r-3)\quad
{\text{if $4\leq r\leq 6$}}\\
8(r-2)^3\quad {\text{if $7\leq r\leq 11$}}\\
2^{r+1}\quad {\text{if $r\geq 12$}}.
\end{cases}
$$

\smallskip
(ii) When $r-1\leq s\leq 2r-4$ define:
$$
G^*(r,d,s,\pi,p):={\binom{m}{2}}s+m(\epsilon+\pi)-p
+\,\max\left(0,\left[\frac{2\pi-(s-1-\epsilon)}{2}\right]\right)
$$
(square brackets indicate the integer part). Even if
$G^*(r,d,s,\pi,p)$ does not depend on $r$, we prefer to use this
notation in order to recall that $r-1\leq s\leq 2r-4$. Observe
that the number $G^*(r,d,s,\pi_0,0)$ is the quoted bound
$G(r,d,s)$ determined in \cite{EH}, Theorem (3.22), when $r-1\leq
s\leq 2r-3$ (in \cite{EH} these numbers are denoted by
$\pi_{\alpha}(d,r)$, with $\alpha:=s-r+2$).

\smallskip
(iii) We define the  numerical function $h_{r,d,s,\pi}$ as
follows:
$$
h_{r,d,s,\pi}(i):=
\begin{cases}
1-\pi+is-\max\left(0,\,\pi_0-\pi-i+1\right)\quad
{\text{if $1\leq i\leq m$}}\\
d-\max\left(0,\left[\frac{2\pi-(s-1-\epsilon)}{2}\right]\right)\quad
{\text{if $i=m+1$}}\\ d\quad {\text{if $i\geq m+2$}}.
\end{cases}
$$

\smallskip
(iv) For a projective subscheme $X\subseteq\mathbb{P}^N$ we will
denote by $\mathcal{I}_X$ its ideal sheaf in $\mathbb{P}^N$, and
by $M(X):=\oplus_{i\in\mathbb{Z}} H^1(\mathbb{P}^{N},\mathcal
I_{X}(i))$ the Hartshorne-Rao Module. We will denote by $h_X$ the
Hilbert function of $X$ \cite{EH}, and by $\Delta h_X$ the first
difference of $h_X$, i.e. $\Delta h_X(i):=h_X(i)-h_X(i-1)$. We
will say that $X$ is a.C.M. if it is arithmetically
Cohen-Macaulay.

\smallskip
(v) Given numerical functions $h_1:\mathbb{Z}\to\mathbb{Z}$ and
$h_2:\mathbb{Z}\to\mathbb{Z}$, we say that $h_1>h_2$ if
$h_1(i)\geq h_2(i)$ for any $i\in\mathbb{Z}$, and if there exists
some $i$ such that $h_1(i)>h_2(i)$.

\end{notations}

\medskip
Our main result is the following:

\medskip
\begin{theorem} \label{first} Fix integers $r,d,s,\pi$ with $r\geq 4$,
$r-1\leq s \leq 2r-4$, $0\leq\pi\leq \pi_0$ and $d> d_0(r)$. Let
$C\subset {\Ps^r}$ be an irreducible, reduced, nondegenerate,
projective curve of degree $d$, and arithmetic genus $p_a(C)$. Let
$\Gamma\subset\mathbb{P}^{r-1}$ be the general hyperplane section
of $C$, and $h_{\Gamma}$ its Hilbert function. Assume that $C$ is
not contained in any surface of degree $<s$, and not contained in
any surface of degree $s$ with sectional genus $>\pi$. Then one
has:

\smallskip (a) $\,\,h_{\Gamma}(i)\geq h_{r,d,s,\pi}(i)$ for any
$i\in\mathbb{Z}$;

\smallskip (b)
$\,\,p_a(C)\leq
G^*\left(r,d,s,\pi,-\binom{\pi_0-\pi+1}{2}\right)$, and therefore
$$G(r,d,s,\pi)\leq
G^*\left(r,d,s,\pi,-\binom{\pi_0-\pi+1}{2}\right);$$

\smallskip (c) $\,$the bound is sharp, and the curves with maximal genus are
a. C.M. with $h_{\Gamma}= h_{r,d,s,\pi}$, and contained in
surfaces $S$ of degree $s$, sectional genus $\pi$ and arithmetic
genus $p_a(S)=-\binom{\pi_0-\pi+1}{2}$.
\end{theorem}
\medskip

By property (c), combined with Corollary \ref{arithmeticgenus2}
below, we see that curves with {\it maximal} arithmetic genus lie
in surfaces with {\it minimal} arithmetic genus. The proof of this
fact relies on a general bound (see Proposition \ref{prop} below)
which, as far as we know, although elementary,  seems to have
escaped explicit notice. We hope that Proposition \ref{prop} can
be useful to obtain further information in the range $s\geq 2r-3$.

As for the other properties, the proof of Theorem \ref{first}
follows a now classic pattern in Castelnuovo-Halphen Theory (see
\cite{EH}), taking into account \cite{GL} which allows us to
estimate the  Hartshorne-Rao module of the general hyperplane
section of an integral surface $S\subset\mathbb{P}^r$ of degree
$r-1\leq s\leq 2r-4$ (compare also with \cite{Park}).

According to the above, previous Theorem \ref{first} suggests a
more refined analysis: {\it given integers $r,d,s,\pi,p$, find the
maximal genus $G(r,d,s,\pi,p)$ for an integral curve in
$\mathbb{P}^r$ of given degree $d\gg s$, not contained in any
surface of degree $<s$, and not contained in any surface of degree
$s$ with sectional genus $>\pi$ and arithmetic genus $<p$}. By
Theorem \ref{first} we already know that when $r-1\leq s \leq
2r-4$ then
$G\left(r,d,s,\pi,-\binom{\pi_0-\pi+1}{2}\right)=G^*\left(r,d,s,\pi,-\binom{\pi_0-\pi+1}{2}\right)$.
To this purpose  we are able to prove the following partial
result.

\medskip
\begin{proposition} \label{second} Fix integers $r,d,s,\pi,p$ with $r\geq 4$,
$r-1\leq s \leq 2r-4$, $0\leq\pi\leq \pi_0$,
$-\binom{\pi_0-\pi+1}{2}\leq p\leq 0$, and $d>d_0(r)$. Let
$C\subset {\Ps^r}$ be an irreducible, reduced, nondegenerate,
projective curve of degree $d$, and arithmetic genus $p_a(C)$.
Assume that $C$ is not contained in any surface of degree $<s$,
and not contained in any surface of degree $s$ with sectional
genus $>\pi$ and with arithmetic genus $<p$. Then one has:

\medskip (a)
$p_a(C)\leq G^*\left(r,d,s,\pi,p\right)$, i.e.
$G\left(r,d,s,\pi,p\right)\leq G^*\left(r,d,s,\pi,p\right)$;

\medskip (b) if the bound is sharp, i.e. if
$G\left(r,d,s,\pi,p\right)= G^*\left(r,d,s,\pi,p\right)$, then the
curves with maximal genus are contained in surfaces of degree $s$,
sectional genus $\pi$ and arithmetic genus $p$;

\medskip (c) if  there is a nondegenerate, irreducible, smooth
curve $\Sigma\subset \mathbb{P}^{r-1}$ of degree $s$ and genus
$\pi$ with the Hartshorne-Rao module  of dimension $-p$ (in this
case one has a fortiori $p\leq -(\pi_0-\pi)$), then the bound is
sharp, and there are extremal a.C.M. curves on the cone
$S\subset\mathbb{P}^r$ over $\Sigma$ (when $2\pi\geq s-1+\epsilon$
we must also assume that $\Sigma$ is an isomorphic projection of a
Castelnuovo curve $\Sigma'\subset \mathbb{P}^{r-1+\pi_0-\pi}$
contained in a smooth rational normal scroll surface);

\medskip (d) when  $p=-\binom{\pi_0-\pi+1}{2}$ or
$p=-(\pi_0-\pi)$ or $p=0$, then bound is sharp.
\end{proposition}
\medskip

The line of the proof is similar to the proof of Theorem
\ref{first}. However  we are forced to slightly  modify it
because, in this more general setting, there is no a minimal
Hilbert function for the general hyperplane section of $C$ as in
Theorem \ref{first}, (a) (and in fact there are extremal curves
which are not a.C.M. (see Remark \ref{stima} below, (iii), (iv)
and (v))). We are able to overcome this difficulty thanks to the
quoted Proposition \ref{prop}. As far as we know, the question of
the existence of a curve $\Sigma$ as in (c) of Proposition
\ref{second} (essentially of a curve with a prescribed
Hartshorne-Rao module) is quite difficult. Therefore previous
proposition appears as a partial result in this setting.

We refer to Remark \ref{stima} for other examples and comments on
extremal curves in the sense of Theorem \ref{second}. Our
assumption $d>d_0(r)$ is certainly not the best possible. It is
only of the simplest form we were able to conceive. For $r\geq 12$
this assumption coincides with the one introduced in \cite{EH}, p.
117, Theorem (3.22).

\bigskip
\section{Preliminary results}

In this section we collect some preliminary results which we need
in order to prove the announced results. We start with the
following consequence of Theorem 1 in \cite{GL} (compare also with
Theorem 3.2 in \cite{Park}). We keep  the notation introduced
before.

\medskip
\begin{proposition} \label{gl}
Let $\Sigma\subset\mathbb{P}^{r-1}$ be a non degenerate integral
curve of degree $s$ with arithmetic genus $\pi$. Assume that
$r-1\leq s\leq 2r-4$. Then one has $h^1(\Sigma,\mathcal
O_{\Sigma}(i))=0$ for any $i\geq 1$,
$h^1(\mathbb{P}^{r-1},\mathcal I_{\Sigma}(1))=\pi_0-\pi$, and
$h^1(\mathbb{P}^{r-1},\mathcal I_{\Sigma}(i))\leq
\max\left(0,\,h^1(\mathbb{P}^{r-1},\mathcal
I_{\Sigma}(i-1))-1\right)$ for any $i\geq 2$. In particular
$h^1(\mathbb{P}^{r-1},\mathcal I_{\Sigma}(i))\leq
\max\left(0,\,\pi_0-\pi+1-i\right)$ for any $i\geq 1$.
\end{proposition}
\medskip

\begin{proof} Since $\pi\leq \pi_0=s-r+1$ then $2\pi-2<s$.
Therefore $\Sigma$ is non special and so $h^1(\Sigma,\mathcal
O_{\Sigma}(i))=0$ for any $i\geq 1$. In particular
$h^1(\mathbb{P}^{r-1},\mathcal I_{\Sigma}(1))=h^0(\Sigma,\mathcal
O_{\Sigma}(1))-h_{\Sigma}(1)=(1-\pi+s)-r=\pi_0-\pi$. It remains to
prove that $h^1(\mathbb{P}^{r-1},\mathcal I_{\Sigma}(i))\leq$
$\max\left(0,h^1(\mathbb{P}^{r-1},\mathcal
I_{\Sigma}(i-1))-1\right)$ for any $i\geq 2$.

To this purpose let $H$ be the general hyperplane section of
$\Sigma$. Since $r-1\leq s\leq 2r-4$ then by Castelnuovo Theory
\cite{EH} we know that $h_H(i)=s$ for any $i\geq 2$, and so
$h^1(\mathbb{P}^{r-2},\mathcal I_{H}(i))=0$ for any $i\geq 2$.
Therefore for any $i\geq 2$ we have the following exact sequence:
\begin{equation}\label{exsq}
0\to
H^0(\Ps^{r-1},\mathcal I_{\Sigma}(i-1))\to H^0(\Ps^{r-1},\mathcal
I_{\Sigma}(i))\to  H^0(\Ps^{r-2},\mathcal
I_{H}(i))\to\quad\quad\quad\quad\quad
\end{equation}
$$ \quad\quad\quad\quad\quad\quad\quad\quad\quad\quad\quad\quad\quad\quad
H^1(\Ps^{r-1},\mathcal I_{\Sigma}(i-1))\to H^1(\Ps^{r-1},\mathcal
I_{\Sigma}(i))\to 0.
$$
In particular we have $h^1(\Ps^{r-1},\mathcal I_{\Sigma}(i))\leq
h^1(\Ps^{r-1},\mathcal I_{\Sigma}(i-1))$. Now suppose by
contradiction that $h^1(\Ps^{r-1},\mathcal I_{\Sigma}(i))=
h^1(\Ps^{r-1},\mathcal I_{\Sigma}(i-1))>0$ for some $i\geq 2$.
Then the map $H^0(\Ps^{r-1},\mathcal I_{\Sigma}(i))\to
H^0(\Ps^{r-2},\mathcal I_{H}(i))$ should be surjective. But by
\cite{GL} we know that the homogeneous ideal of $H$ is generated
by quadrics. It would follow that the map $H^0(\Ps^{r-1},\mathcal
I_{\Sigma}(j))\to H^0(\Ps^{r-2},\mathcal I_{H}(j))$ is onto for
any $j\geq i$, which in turn would imply that
$h^1(\Ps^{r-1},\mathcal I_{\Sigma}(j))= h^1(\Ps^{r-1},\mathcal
I_{\Sigma}(i-1))>0$ for any $j\geq i-1$. This is absurd.
\end{proof}

\medskip
\begin{lemma} \label{num}
With the same notation as above we have:

\smallskip
(1)
$\sum_{i=1}^{+\infty}(d-h_{r,d,s,\pi}(i))=G^*\left(r,d,s,\pi,-\binom{\pi_0-\pi+1}{2}\right)$;

\smallskip
(2) if $\pi'<\pi$ then $h_{r,d,s,\pi'}> h_{r,d,s,\pi}$, therefore
$G^*\left(r,d,s,\pi',-\binom{\pi_0-\pi'+1}{2}\right)<G^*\left(r,d,s,\pi,-\binom{\pi_0-\pi+1}{2}\right)$;

\smallskip
(3) if $d\geq (2s+1)(s+1)$ then $h_{r,d,s+1,\pi_0'}>
h_{r,d,s,\pi}$, therefore $G^*\left(r,d,s+1,\pi_0',0\right)$ $
<G^*\left(r,d,s,\pi,-\binom{\pi_0-\pi+1}{2}\right)$ (here we set
$\pi_0':=\pi_0(s+1,r-1)=(s+1)-r+1$).
\end{lemma}
\medskip
The proof is straightforward, and so we omit it.

\medskip
\begin{lemma} \label{ons}
Fix integers $r,d,s$ with $r\geq 4$, $r-1\leq s \leq 2r-4$, and
$d\geq s(s-1)$. Let $C\subset {\Ps^r}$ be an irreducible, reduced,
non degenerate, projective curve of degree $d$, with general
hyperplane section $\Gamma$. Assume that $C$ is contained in an
integral projective surface $S\subset {\Ps^r}$ of degree $s$ and
sectional genus $\pi$. Then one has $h_{\Gamma}(i)\geq
h_{r,d,s,\pi}(i)$ for any $i\geq 1$.
\end{lemma}
\medskip

\begin{proof} By Bezout Theorem we have
$h_{\Gamma}(i)=h_{\Sigma}(i)$ for any $1\leq i\leq m$, where
$\Sigma$ denotes the general hyperplane section  of $S$. On the
other hand, by Proposition \ref{gl},  for $1\leq i\leq m$ one has
$$
h_{\Sigma}(i)=h^0(\Sigma,\mathcal
O_{\Sigma}(i))-h^1(\Ps^{r-1},\mathcal
I_{\Sigma}(i))=1-\pi+is-h^1(\Ps^{r-1},\mathcal I_{\Sigma}(i))
$$
$$
\geq 1-\pi+is-\max\left(0,\,\pi_0-\pi-i+1\right)=h_{r,d,s,\pi}(i).
$$
It remains to examine the range $i\geq m+1$.

\smallskip To this purpose first notice that if $L$ is a general
hyperplane such that $\Sigma=S\cap L$, then
$Sing(\Sigma)=Sing(S)\cap L$, and so
$\deg(Sing(S))=\deg(Sing(\Sigma))\leq \pi_0$. It follows that $C$
is not contained in $Sing(S)$ because $d\gg s$. Hence $\Gamma$
does not meet the singular locus of $\Sigma$, i.e. $\Gamma\subset
\Sigma \backslash Sing(\Sigma)$, and so $\Gamma$ defines an
effective Cartier divisor on $\Sigma$. It follows the existence of
the exact sequence:
$$
0\to\mathcal{O}_{\Sigma}(-\Gamma+(m+j)H)\to
\mathcal{O}_{\Sigma}((m+j)H)\to\mathcal{O}_{\Gamma}\to 0,
$$
where $H$ denotes the general hyperplane section of $\Sigma$.
Since $d\gg s$ then from \cite{GLP} it follows that the natural
map $ H^0(\Ps^{r-1}, \mathcal{O}_{\Ps^{r-1}}(m+j))\to H^0(\Sigma,
\mathcal{O}_{\Sigma}(m+j)) $ is surjective for any $j\geq 0$, and
so from previous exact sequence we get:
\begin{equation}\label{ex}
h_{\Gamma}(m+j)=h^0(\Sigma, \mathcal{O}_{\Sigma}(m+j))-h^0(\Sigma,
\mathcal{O}_{\Sigma}(-\Gamma+(m+j)H))
\end{equation}
$$
= 1-\pi+(m+j)s-h^0(\Sigma, \mathcal{O}_{\Sigma}(-\Gamma+(m+j)H)).
$$
If $h^1(\Sigma, \mathcal{O}_{\Sigma}(-\Gamma+(m+j)H))=0$ then
$h^0(\Sigma, \mathcal{O}_{\Sigma}(-\Gamma+(m+j)H))=1-\pi+(m+j)s-d$
and therefore $h_{\Gamma}(m+j)=d$. Otherwise $h^1(\Sigma,
\mathcal{O}_{\Sigma}(-\Gamma+(m+j)H))>0$ and by Clifford's Theorem
for possibly singular curves (see \cite{EH}, p. 46, or
\cite{Fujita}, Proposition 1.5., and compare with \cite{EH},
p.121) we know that
$$
h^0(\Sigma, \mathcal{O}_{\Sigma}(-\Gamma+(m+j)H))-1\leq
\frac{(m+j)s-d}{2}
$$
hence
$$
h_{\Gamma}(m+j)\geq \frac{(m+j)s+d}{2}-\pi,
$$
and so $h_{\Gamma}(i)\geq h_{r,d,s,\pi}(i)$ for any $i\geq m+1$.
\end{proof}

\medskip
\begin{lemma}
\label{arithmeticgenus} Let $S\subset \Ps^{r}$  be an irreducible,
reduced, non degenerate projective surface of degree $s$ and
arithmetic genus $p_a(S)$. Denote by $\Sigma$ the general
hyperplane section of $S$, and by $H$ the general hyperplane
section of $\Sigma$. For any integer $i$ set $\delta_i:=\Delta
h_{\Sigma}(i)-h_H(i)$ and $ \mu_i:=\Delta h_{S}(i)-h_{\Sigma}(i)$.
Then we have:
$$
p_a(S)=\sum_{i=1}^{+\infty}(i-1)(s-h_{H}(i))-
\sum_{i=1}^{+\infty}(i-1)\delta_i+\sum_{i=1}^{+\infty}\mu_i.
$$
In particular, when $r-1\leq s\leq 2r-4$, then
$$
p_a(S)=-\dim_{\mathbb{C}} M(\Sigma) +\sum_{i=1}^{+\infty}\mu_i,
$$
where $M(\Sigma)$ denotes the Hartshorne-Rao module of $\Sigma$.
\end{lemma}
\medskip

\begin{remark}\label{ciro}
By \cite{roccadipapa}, p. 30, we know that
$$
\delta_i=\dim_{\mathbb{C}}\left[Ker\left(H^1(\mathbb{P}^{r-1},\mathcal{I}_{\Sigma}(i-1))
\to H^1(\mathbb{P}^{r-1},\mathcal{I}_{\Sigma}(i))\right)\right].
$$
Similarly as in \cite{roccadipapa}, p. 30 one may prove that
$$
\mu_i=\dim_{\mathbb{C}}\left[Ker\left(H^1(\mathbb{P}^{r},\mathcal{I}_{S}(i-1))
\to H^1(\mathbb{P}^{r},\mathcal{I}_{S}(i))\right)\right].
$$
\end{remark}
\medskip

\begin{proof}[Proof of Lemma \ref{arithmeticgenus}]
Recall that when $t\gg 0$ then the Hilbert polynomial of $S$ at
level $t$ coincides with the Hilbert function $h_S(t)$ of $S$.
Therefore we have:
\begin{equation}\label{genusofs}
p_a(S)=h_S(t)-s{\binom {t+1}{2}}+t\pi-t-1,
\end{equation}
where $\pi$ denotes the sectional genus of $S$. Now we may write:
$$
h_S(t)=\sum_{j=0}^{t}\Delta h_S(j)=
\sum_{j=0}^{t}h_{\Sigma}(j)+\mu_j=
\sum_{j=0}^{t}\left(\sum_{i=0}^{j}\Delta
h_{\Sigma}(i)\right)+\sum_{j=0}^{t}\mu_j
$$
$$
=
\sum_{j=0}^{t}\left(\sum_{i=0}^{j}h_H(i)+\delta_i\right)+\sum_{j=0}^{t}\mu_j=
\sum_{i=0}^{t}(t-i+1)(h_H(i)+\delta_i)+\sum_{j=0}^{t}\mu_j.
$$
Taking into account that $\delta_0=\mu_0=0$ and that $h_H(0)=1$,
inserting previous equality into (\ref{genusofs}) we obtain:
\begin{equation}\label{genusofss}
p_a(S)=t\left[\pi + \sum_{i=1}^{t} h_H(i)+\delta_i\right]
-\sum_{i=1}^{t}(i-1)(h_H(i)+\delta_i)+\sum_{j=0}^{t}\mu_j
-s{\binom {t+1}{2}}.
\end{equation}
By \cite{roccadipapa}, pg. $31$, we have (recall that $t\gg 0$) $
\pi = \sum_{i=1}^{t} \left(s-h_H(i)-\delta_i\right)$, therefore
from (\ref{genusofss}) it follows that
$$
p_a(S)=\left[t^2-{\binom {t+1}{2}}\right]s
-\sum_{i=1}^{t}(i-1)(h_H(i)+\delta_i)+\sum_{j=0}^{t}\mu_j
$$
$$
 = \sum_{i=1}^{+\infty}(i-1)(s-h_{H}(i))-
\sum_{i=1}^{+\infty}(i-1)\delta_i+\sum_{i=1}^{+\infty}\mu_i.
$$

As for the last claim, observe that when $r-1\leq s\leq 2r-4$ we
have $h_H(i)=s$ for any $i\geq 2$ by Castelnuovo Theory \cite{EH},
and so $\sum_{i=1}^{+\infty}(i-1)(s-h_{H}(i))=0$. Moreover, by
Remark \ref{ciro} and (\ref{exsq}) we see that
$\delta_i=h^1(\mathbb{P}^{r-1},\mathcal{I}_{\Sigma}(i-1))-h^1(\mathbb{P}^{r-1},\mathcal{I}_{\Sigma}(i))$,
from which we get
$\sum_{i=1}^{+\infty}(i-1)\delta_i=\dim_{\mathbb{C}}M(\Sigma)$.
\end{proof}

\medskip
\begin{corollary}
\label{arithmeticgenus2} Let $S\subset \Ps^{r}$  be an
irreducible, reduced, non degenerate projective surface of degree
$s$, sectional genus $\pi$, and arithmetic genus $p_a(S)$. Assume
that $r-1\leq s\leq 2r-4$. Then we have
$$
-\binom{\pi_0-\pi+1}{2}\leq p_a(S)\leq 0.
$$
\end{corollary}
\medskip

\begin{proof} By previous Lemma \ref{arithmeticgenus} and Proposition
\ref{gl} we deduce
$$
p_a(S)\geq - \dim_{\mathbb{C}}M(\Sigma)\geq
-\binom{\pi_0-\pi+1}{2}.
$$
Therefore we only have to prove that $p_a(S)\leq 0$. To this aim
first observe that $p_a(S)=-h^1(S,\mathcal O_S)+h^2(S,\mathcal
O_S)\leq h^2(S,\mathcal O_S)$. Moreover by \cite{NV}, Lemma 5, we
know that $h^2(S,\mathcal O_S)\leq
\sum_{i=1}^{+\infty}(i-1)(s-h_{H}(i))$. This number is $0$ because
$h_{H}(i)=s$ for any $i\geq 2$. Hence $p_a(S)\leq 0$.
\end{proof}

\medskip
\begin{remark}
With the same assumption as in Corollary \ref{arithmeticgenus2},
previous argument proves  that $p_a(S)=-\binom{\pi_0-\pi+1}{2}$ if
and only if $M(S)=0$, and
$h^1(\Ps^{r-1},\mathcal{I}_{\Sigma}(i))=\max\left(0,\,\pi_0-\pi-i+1\right)$
for any $i\geq 1$.
\end{remark}
\medskip

\smallskip Next lemma, for which we did not succeed in finding an appropriate reference,
states an explicit upper bound for Castelnuovo-Mumford regularity
of an integral projective surface. We need it in order to make
explicit the assumption $d\gg s$ appearing in Proposition
\ref{prop} below (which in turn we will use, via Corollary
\ref{essesmall}, in the proof of Theorem \ref{first}, (c), and
Proposition \ref{second}, (a)).

\medskip
\begin{lemma}
\label{regularity} Let $S\subset \Ps^{r}$  be an irreducible,
reduced, non degenerate projective surface of degree $s\geq
r-1\geq 2$ and Castelnuovo-Mumford regularity $reg(S)$. Then one
has
$$
reg(S)\leq (s-r+2)\left(\frac{s^2}{2(r-2)}+1\right)+1.
$$
\end{lemma}
\medskip

\begin{proof} Let $\Sigma$ be the general hyperplane section of
$S$. By \cite{GLP} we know that:
\begin{equation}\label{GLP}
reg(\Sigma)\leq s-r+3.
\end{equation}
Hence, by (\cite{Mumf}, p. 102) we have
$$
reg(S)\leq s-r+3+h^1(\Ps^{r}, \mathcal{I}_S(s-r+2)).
$$
Therefore it suffices to prove that:
\begin{equation}\label{accauno}
h^1(\Ps^{r}, \mathcal{I}_S(s-r+2))\leq (s-r+2)\frac{s^2}{2(r-2)}.
\end{equation}

To this purpose first notice that by (\ref{GLP}) we know that
$h^1(\Ps^{r-1}, \mathcal{I}_{\Sigma}(i))=0$ for any $i\geq s-r+2$,
so the natural map $H^0(\Ps^{r}, \mathcal{O}_{\Ps^r}(i))\to
H^0(\Sigma, \mathcal{O}_{\Sigma}(i))$ is surjective for any $i\geq
s-r+2$. A fortiori the natural map $H^0(\Ps^{r},
\mathcal{O}_{S}(i))\to H^0(\Ps^{r}, \mathcal{O}_{\Sigma}(i))$ is
surjective for any $i\geq s-r+2$. It follows that $H^1(S,
\mathcal{O}_{S}(i-1))\subseteq H^1(S, \mathcal{O}_{S}(i))$ for any
$i\geq s-r+2$ in view of the exact sequence $0\to
\mathcal{O}_{S}(i-1)\to \mathcal{O}_{S}(i)\to
\mathcal{O}_{\Sigma}(i)\to 0$, and from the vanishing $H^1(S,
\mathcal{O}_{S}(i))=0$ for $i\gg 0$ we obtain $ H^1(S,
\mathcal{O}_{S}(i))=0\quad {\text{for any $i\geq s-r+1$}}$. Hence
we have:
\begin{equation}\label{sommap}
h^1(\Ps^{r}, \mathcal{I}_S(s-r+2))=h^0(S,
\mathcal{O}_S(s-r+2))-h_S(s-r+2)\leq p_S(s-r+2)-h_S(s-r+2),
\end{equation}
where $h_S(s-r+2)$ and $p_S(s-r+2)$ denote the Hilbert function
and the Hilbert polynomial of $S$ at level $s-r+2$. By \cite{EH},
Lemma (3.1), we may estimate
$$
h_S(s-r+2)\geq \sum_{i=0}^{s-r+2}h_{\Sigma}(i)\geq
\sum_{i=0}^{s-r+2}\left[\sum_{j=0}^{i}h_H(j)\right]=
\sum_{i=0}^{s-r+2}(s-r+3-i)h_H(i),
$$
where $h_H$ denotes the Hilbert function of the general hyperplane
section $H$ of $\Sigma$. Since
$$
p_S(t)=s{\binom{t+1}{2}}+(1-\pi)t+1+p_a(S)
$$
($\pi$ and $p_a(S)$ denote the sectional and the arithmetic genus
of $S$) from (\ref{sommap}) it follows that:
\begin{equation}\label{accadue}
h^1(\Ps^{r}, \mathcal{I}_S(s-r+2))\leq
\left[s{\binom{s-r+3}{2}}+(1-\pi)(s-r+2)+1+p_a(S)\right]
\end{equation}
$$
-\left[\sum_{i=0}^{s-r+2}(s-r+3-i)h_H(i)\right]
=p_a(S)-\sum_{i=1}^{s-r+3}(i-1)(s-h_H(i))
$$
$$
+ 2s{\binom{s-r+3}{2}}-(s-r+2)(\pi+\sum_{i=1}^{s-r+3}h_H(i)).
$$
From \cite{NV}, Lemma 5, we know that:
$$
p_a(S)=h^2(S, \mathcal{O}_S)-h^1(S, \mathcal{O}_S)\leq h^2(S,
\mathcal{O}_S)\leq \sum_{i=1}^{+\infty}(i-1)(s-h_H(i)),
$$
and from \cite{EH}, Theorem (3.7),  we have:
$$
\sum_{i=1}^{+\infty}(i-1)(s-h_H(i))=\sum_{i=1}^{s-r+3}(i-1)(s-h_H(i))
$$
because $h_H(i)=s$ for $i\geq w+1$, and $w+1 \leq s-r+4$ (we
define $w$ by dividing $s-1=w(r-2)+v$, $0\leq v\leq r-3$). We
deduce that:
$$
p_a(S)-\sum_{i=1}^{s-r+3}(i-1)(s-h_H(i))\leq 0,
$$
and so from (\ref{accadue}) we get:
$$
h^1(\Ps^{r}, \mathcal{I}_S(s-r+2))\leq
2s{\binom{s-r+3}{2}}-(s-r+2)(\pi+\sum_{i=1}^{s-r+3}h_H(i))
$$
$$
=(s-r+2)\left[\left(\sum_{i=1}^{s-r+3}s-h_H(i)\right)-\pi\right].
$$
By \cite{EH}, Corollary (3.3) and proof, and Theorem (3.7), the
term $\sum_{i=1}^{s-r+3}(s-h_H(i))$ is bounded by Castelnuovo's
bound $G(r-1,s):={\binom{w}{2}}(r-2)+wv$ for the arithmetic genus
of $\Sigma$. Since $G(r-1,s)\leq \frac{s^2}{2(r-2)}$ we get
$$
\sum_{i=1}^{s-r+3}(s-h_H(i))\leq \frac{s^2}{2(r-2)}.
$$
Combining the last two estimates we obtain (\ref{accauno}).
\end{proof}

\medskip
\begin{proposition}
\label{prop} Let $S\subset \Ps^{r}$  be an irreducible, reduced,
non degenerate projective surface of degree $s\geq r-1\geq 2$,
sectional genus $\pi$ and arithmetic genus $p_a(S)$. Let $C\subset
S$ be an irreducible, reduced, non degenerate projective curve of
degree $d\geq \frac{s^4}{2(r-2)}$. Denote by $p_a(C)$, by
$\mathcal{I}_{C}\subset \mathcal{O}_{\Ps^r}$ and by $h_C$  the
arithmetic genus, the ideal sheaf  and the Hilbert function of
$C$. Denote by $\Gamma$ the general hyperplane section of $C$ and
by $h_{\Gamma}$ its Hilbert function. Then one has:
\begin{equation}\label{slprop}
p_a(C)= {\binom{m}{2}}s+m(\epsilon+\pi)-p_a(S)+
\sum_{i=m+1}^{+\infty} d-\Delta h_{C}(i).
\end{equation}
In particular one has
\begin{equation}\label{lprop}
p_a(C)\leq {\binom{m}{2}}s+m(\epsilon+\pi)-p_a(S)+
\sum_{i=m+1}^{+\infty} d-h_{\Gamma}(i),
\end{equation}
and $p_a(C)$ attains this bound if and only if
$h^1(\Ps^{r},\mathcal{I}_{C}(i))=0$ for any $i\geq m$.
\end{proposition}
\medskip

\begin{proof} Since for $t\gg 0$
we have $h_C(t)=1-p_a(C)+dt$ then we may write
\begin{equation}\label{sslprop}
p_a(C)=dt+1-h_C(t)=\sum_{i=1}^{t}d-\Delta
h_{C}(i)=\sum_{i=1}^{+\infty}d-\Delta h_{C}(i)
\end{equation}
$$
=\sum_{i=1}^{m}d-\Delta h_{C}(i)+\sum_{i=m+1}^{+\infty}d-\Delta
h_{C}(i).
$$
On the other hand by Bezout's Theorem we have $h_C(i)=h_S(i)$ for
any $i\leq m$, and therefore we have
$$
\sum_{i=1}^{m}d-\Delta h_{C}(i)=md+1-h_C(m)=md+1-h_S(m).
$$
By Lemma \ref{regularity} we deduce that $h_S(m)$ coincides with
the Hilbert polynomial $p_S(m)$ of $S$ at level $m$, i.e.
$$
h_S(m)=p_S(m)={\binom{m+1}{2}}s+m(1-\pi)+1+p_a(S).
$$
It follows that
$$
\sum_{i=1}^{m}d-\Delta h_{C}(i)=md+1-h_S(m)
$$
$$
=md+1-\left[{\binom{m+1}{2}}s+m(1-\pi)+1+p_a(S)\right]={\binom{m}{2}}s+m(\epsilon+\pi)-p_a(S).
$$
Inserting this into (\ref{sslprop}) we obtain (\ref{slprop}).

As for (\ref{lprop}), we observe that
$$
\sum_{i=m+1}^{+\infty} d-\Delta
h_{C}(i)=\sum_{i=m+1}^{+\infty}d-h_{\Gamma}(i)
-\sum_{i=m+1}^{+\infty}\Delta h_C(i)-h_{\Gamma}(i).
$$
Hence (\ref{slprop}) implies (\ref{lprop}) because $\Delta
h_C(i)-h_{\Gamma}(i)\geq 0$ for any $i$ (\cite{EH}, Lemma (3.1)).
Moreover we deduce that $p_a(C)$ attains the bound appearing in
(\ref{lprop}) if and only if $\sum_{i=m+1}^{+\infty}\Delta
h_C(i)-h_{\Gamma}(i)=0$. And this is equivalent to say that
$h^1(\Ps^{r},\mathcal{I}_{C}(i))=0$ for any $i\geq m$ in view of
Remark \ref{ciro}. This concludes the proof of Proposition
\ref{prop}.
\end{proof}

\medskip
\begin{remark}\label{ms}
(i) From the proof it follows that {\it if there is an a.C.M.
curve on $S$ of degree $d\gg s$ then
$\sum_{i=1}^{+\infty}\mu_i=0$, and therefore the Hartshorne-Rao
module of $S$ vanishes}.

\smallskip
(ii) When $S$ is smooth one knows that $reg(S)\leq s-r+3$
\cite{Laz}, and so to prove Proposition \ref{prop} one may simply
assume that $m\geq s-r+2$, or also $d\geq s(s-r+3)$. This last
numerical assumption is enough also if one knows that
$h^1(\mathbb{P}^r,\mathcal I_S(m))=0$, e.g. when $S$ is a. C. M..
\end{remark}
\medskip

Combining (\ref{lprop}) with Lemma \ref{ons} we get the following

\medskip
\begin{corollary}
\label{essesmall} Let $S\subset \Ps^{r}$  be an irreducible,
reduced, non degenerate projective surface of degree $s$ with
$2\leq r-1\leq s\leq 2r-4$, sectional genus $\pi$ and arithmetic
genus $p_a(S)$. Let $C\subset S$ be an irreducible, reduced, non
degenerate projective curve of arithmetic genus $p_a(C)$ and
degree $d\geq \frac{s^4}{2(r-2)}$. Then one has:
\begin{equation}\label{lessesmall}
p_a(C)\leq G^*(r,d,s,\pi,p_a(S)).
\end{equation}
\end{corollary}
\medskip

\bigskip
\section{Proof of Theorem \ref{first} and of Proposition \ref{second}}

We begin by proving Theorem \ref{first}.

\smallskip
(a) First assume $C$ is not contained in any surface of degree
$s$. Then $C$ is not contained in any surface of degree $<s+1$. By
\cite{EH} we know that $h_{\Gamma}(i)\geq h_{r,d,s+1,\pi'_0}(i)$
for any $i$, and by Lemma \ref{num} we deduce $h_{\Gamma}(i)\geq
h_{r,d,s,\pi}(i)$ for any $i$. Hence we may assume that $C$ is
contained in a surface  of degree $s$, with sectional genus
$\pi'\leq \pi$. By Lemma \ref{num} and by Lemma \ref{ons} we get
again $h_{\Gamma}(i)\geq h_{r,d,s,\pi}(i)$ for any $i$.

\smallskip
(b) Since in general we have $p_a(C)\leq
\sum_{i=1}^{+\infty}(d-h_{\Gamma}(i))$ (\cite{EH}, Corollary
(3.2)) then by (a) and Lemma \ref{num} we deduce
$$
p_a(C)\leq \sum_{i=1}^{+\infty}(d-h_{\Gamma}(i))\leq
\sum_{i=1}^{+\infty}(d-h_{r,d,s,\pi}(i))=G^*\left(r,d,s,\pi,-\binom{\pi_0-\pi+1}{2}\right).
$$

\smallskip
(c) If the bound is sharp, i.e. if $p_a(C)=
G^*\left(r,d,s,\pi,-\binom{\pi_0-\pi+1}{2}\right)$, then previous
inequality shows that $p_a(C)=
\sum_{i=1}^{+\infty}(d-h_{\Gamma}(i))$, i.e. $C$ is a.C.M., and
$h_{\Gamma}=h_{r,d,s,\pi}$. Moreover, the same argument developed
in (a) and (b), combined with Lemma \ref{num}, proves also that if
$C$ reaches the bound then $C$ must be contained in a surface $S$
of degree $s$ and sectional genus $\pi$. As for $p_a(S)$, observe
that, by Corollary \ref{essesmall}, we have
$$
p_a(C)=G^*\left(r,d,s,\pi,-\binom{\pi_0-\pi+1}{2}\right) \leq
G^*\left(r,d,s,\pi,p_a(S)\right).
$$
It follows  $p_a(S)\leq -\binom{\pi_0-\pi+1}{2}$, and by Corollary
\ref{arithmeticgenus2} we get $p_a(S)= -\binom{\pi_0-\pi+1}{2}$.

Now, to conclude the proof of Theorem \ref{first}, we only have to
prove that the upper bound is sharp.

To this purpose, fix integers $r\geq 4$, $r-1\leq s\leq 2r-4$,
$0\leq \pi <\pi_0:=s-r+1$. Let $\Sigma'\subset
\mathbb{P}^{r-1+\pi_0-\pi}$ be a smooth Castelnuovo curve of
degree $s$ and genus $\pi$ (which we may find on a smooth rational
normal scroll surface in $\mathbb{P}^{r-1+\pi_0-\pi}$ (use
\cite{Hartshorne}, Corollary 2.18 and 2.19)).

Choose general $\pi_0-\pi+2$ points on $\Sigma'$ (compare with
\cite{Park}, p. 13, Example 3.7). Denote by
$\mathbb{P}^{\pi_0-\pi+1}$ the linear space generated by these
points. A general subspace $\mathbb{P}^{\pi_0-\pi-1}\subset
\mathbb{P}^{\pi_0-\pi+1}$ defines a projection
$\varphi:\mathbb{P}^{r-1+\pi_0-\pi}\backslash
\mathbb{P}^{\pi_0-\pi-1}\to \mathbb{P}^{r}$ which maps
isomorphically $\Sigma'$ to a curve $\Sigma\subset
\mathbb{P}^{r-1}$. Since
$\varphi(\mathbb{P}^{\pi_0-\pi+1}\backslash
\mathbb{P}^{\pi_0-\pi-1})$ is a $(\pi_0-\pi+2)-$secant line to
$\Sigma$ then Castelnuovo-Mumford regularity of $\Sigma$ is at
least $\pi_0-\pi+2$. By Lemma \ref{gl} it follows that
$h^1(\mathbb{P}^{r-1},\mathcal I_{\Sigma}(i))=
\max\left(0,\,\pi_0-\pi+1-i\right)$ for any $i\geq 1$ and so
$h_{\Sigma}(i)=1-\pi+si-\max\left(0,\,\pi_0-\pi+1-i\right)$ for
any $i\geq 1$. In particular, once fixed an integer $d\gg s$, we
have $h_{\Sigma}(i)=h_{r,d,s,\pi}(i)$ for $1\leq i\leq m$ (with
$d-1=ms+\epsilon$, $0\leq \epsilon\leq s-1$).

Denote by  $S\subset\mathbb{P}^{r}$  the projective cone on
$\Sigma$. Fix an integer $k\gg s$ of type $k-1=\mu s+\epsilon$,
$0\leq \epsilon\leq s-1$, and a set $D$ of $s-1-\epsilon$ distinct
points on $\Sigma$. Let $C(D)\subset S$ be the cone over $D$, and
let $F\subset \mathbb{P}^r$ be a hypersurface of degree $\mu+1$
containing $C(D)$,  consisting of $\mu+1$ sufficiently general
hyperplanes. Let $R$ be the residual curve to $C(D)$ in the
complete intersection of $F$ with $S$. Equipped with the reduced
structure, $R$ is a cone over $k$ distinct points of $\Sigma$. In
particular $R$ is a (reducible) a.C.M. curve of degree $k$ on $S$,
and, if we denote by $R'$ the general hyperplane section of $R$,
we have $p_a(R)=\sum_{i=1}^{+\infty}(k-h_{R'}(i))$. We make the
following claim. We will prove it in a while.

\medskip
{\bf{Claim.}} {\it For a suitable $D$ one has
$h_{R'}(i)=h_{r,k,s,\pi}(i)$ for any $i\geq 1$.}
\medskip

It follows that
$$
p_a(R)=\sum_{i=1}^{+\infty}(k-h_{R'}(i))=\sum_{i=1}^{+\infty}(k-h_{r,k,s,\pi}(i))=
G^*\left(r,k,s,\pi,-\binom{\pi_0-\pi+1}{2}\right).
$$

Now let $d\gg k$, with $d-1=ms+\epsilon$. Let $G\subset
\mathbb{P}^r$ be a   hypersurface of degree $m+1$ containing
$C(D)$ such that the residual curve $C$ in the complete
intersection of $G$ with $S$, equipped with the reduced structure,
is an integral curve of degree $d$, with a singular point of
multiplicity $k$ at the vertex $p$ of $S$, and   tangent cone at
$p$ equal to $R$. We are going to prove that $C$ is the curve we
are looking for, i.e.
$$p_a(C)=G^*\left(r,d,s,\pi,-\binom{\pi_0-\pi+1}{2}\right).$$
To this aim, let $\widetilde{S}$ be the blowing-up of $S$ at the
vertex. By \cite{Hartshorne}, p. 374, we know that $\widetilde{S}$
is the ruled surface $\mathbb{P}(\mathcal O_{\Sigma}\oplus\mathcal
O_{\Sigma}(-1))\to \Sigma$. Denote by $E$ the exceptional divisor,
by $f$ the line of the ruling, and by $L$ the pull-back of the
hyperplane section. We have $L^2=s$, $L\cdot f=1$, $f^2=0$,
$L\equiv E+sf$ and $K_{\widetilde{S}}\equiv -2L+(2\pi-2+s)f$. Let
$\widetilde{C}\subset \widetilde{S}$ be the blowing-up of $C$ at
$p$, which is nothing but the normalization of $C$. Since $C$ has
degree $d$ then $\widetilde{C}$ belongs to the numerical class of
$(m+1+a)L+(1+\epsilon-(a+1)s)f$ for some integer $a$. Moreover
$E\cdot \widetilde{C}=1+\epsilon-(a+1)s=k$, so
$$
a=-\frac{k+s-1-\epsilon}{s}=-(\mu+1).
$$
By the adjunction formula we get
$$
p_a(\widetilde{C})=\binom{m}{2}s+m(\epsilon+\pi)+\pi-\frac{1}{2}a^2s+a(\pi+\epsilon-\frac{1}{2}s).
$$
On the other hand we have
$$
p_a(C)=p_a(\widetilde{C})+\delta_p
$$
where $\delta_p$ is the delta invariant of the singularity
$(C,p)$. Since the tangent cone of $C$ at $p$ is $R$ then the
delta invariant  is equal to the difference between the arithmetic
genus of $R$ and the arithmetic genus of $k$ disjoint lines in the
projective space, i.e.
$$
\delta_p=p_a(R)-(1-k)=G^*\left(r,k,s,\pi,-\binom{\pi_0-\pi+1}{2}\right)-(1-k).
$$
It follows that
$$
p_a(C)=\binom{m}{2}s+m(\epsilon+\pi)+\pi-\frac{1}{2}a^2s+a(\pi+\epsilon-\frac{1}{2}s)
$$
$$
+ G^*\left(r,k,s,\pi,-\binom{\pi_0-\pi+1}{2}\right)-(1-k).
$$
Taking into account that $a=-\frac{k+s-1-\epsilon}{s}$, a direct
computation proves that this number is exactly
$G^*\left(r,d,s,\pi,-\binom{\pi_0-\pi+1}{2}\right)$.

It remains to prove the claim, i.e. that {\it for a suitable $D$
one has $h_{R'}(i)=h_{r,k,s,\pi}(i)$ for any $i\geq 1$}. This
certainly holds true for any $D$ and  any $1\leq i\leq \mu$
because in this range we have by construction
$h_{\Sigma}(i)=h_{r,d,s,\pi}(i)$, and $h_{R'}(i)=h_{\Sigma}(i)$ by
Bezout Theorem. This holds true also in the range $i\geq \mu+2$ by
degree reasons (compare with the proof of Lemma \ref{ons}). It
remains to examine the case $i=\mu+1$. If
$\max\left(0,\left[\frac{2\pi-(s-1-\epsilon)}{2}\right]\right)=0$
then, as before,  again by degree reasons we have
$h_{R'}(\mu+1)=h_{r,k,s,\pi}(\mu+1)=k$. Otherwise
$\max\left(0,\left[\frac{2\pi-(s-1-\epsilon)}{2}\right]\right)>0$.
In this case let $S'\subset\mathbb{P}^{r+\pi_0-\pi}$ be the cone
over $\Sigma'$. By \cite{CCD}, Example 6.5 (here we need to choose
$\Sigma'$ on a smooth rational normal scroll surface), we know
that for a suitable set $D'$ (in \cite{CCD} denoted by $Z'$) of
$s-1-\epsilon$ distinct points of $\Sigma'$, a general curve $C'$,
obtained from the cone over $D'$ through a linkage with $S'$ and a
hypersurface of degree $\mu+1$, is an integral curve   of degree
$k$ and maximal arithmetic genus
$p_a(C')=G(r+\pi_0-\pi,k,s)=G^*(r+\pi_0-\pi,k,s,\pi,0)$. Let
$\Gamma'$ and $H'$ be the general hyperplane sections of $C'$ and
$\Sigma'$. We have $\mathcal O_{\Sigma'}(D')\cong
\mathcal{O}_{\Sigma'}(-\Gamma'+(\mu+1)H')$. Since $C'$ is maximal
then by a similar computation as in (\ref{ex}) we see that
$h^0(\Sigma',
\mathcal{O}_{\Sigma'}(D'))=s-\epsilon-\pi+d-h_{r+\pi_0-\pi,k,s,\pi}(\mu+1)$.
Since $h_{r+\pi_0-\pi,k,s,\pi}(\mu+1)=h_{r,k,s,\pi}(\mu+1)$ then
$h^0(\Sigma',
\mathcal{O}_{\Sigma'}(D'))=s-\epsilon-\pi+d-h_{r,k,s,\pi}(\mu+1)$.
Therefore if we choose $D$ as the  divisor on $\Sigma$
corresponding to $D'$ via the isomorphism $\Sigma'\cong\Sigma$, as
in (\ref{ex}) we have
$$
h_{R'}(\mu+1)=h^0(\Sigma, \mathcal{O}_{\Sigma}(\mu+1))-h^0(\Sigma,
\mathcal{O}_{\Sigma}(-R'+(\mu+1)H))
$$
$$
= 1-\pi+(\mu+1)s-h^0(\Sigma, \mathcal{O}_{\Sigma}(D)) =
1-\pi+(\mu+1)s-h^0(\Sigma',
\mathcal{O}_{\Sigma}(D'))=h_{r,k,s,\pi}(\mu+1).
$$

This concludes the proof of Theorem \ref{first}.

\medskip
\begin{remark}\label{scroll}
Constructing extremal curves as above, we need to choose $\Sigma'$
on a smooth rational normal scroll surface only in the case
$\max\left(0,\left[\frac{2\pi-(s-1-\epsilon)}{2}\right]\right)>0$,
i.e. when $2\pi\geq s-\epsilon+1$.
\end{remark}
\medskip

Next we turn to the proof of Proposition \ref{second}.

\smallskip
(a)  First assume $C$ is not contained in any surface of degree
$s$. Then $C$ is not contained in any surface of degree $<s+1$. By
\cite{EH} we know that $p_a(C)\leq
G(r,d,s+1)=\frac{d^2}{2(s+1)}+O(d)$ which is strictly less than
$G^*(r,d,s,\pi,p)$ because $d\gg s$ and
$G^*(r,d,s,\pi,p)=\frac{d^2}{2s}+O(d)$. Hence we may assume that
$C$ is contained in a surface  of degree $s$, with sectional genus
$\pi'\leq \pi$. If $\pi'<\pi$ then by Theorem \ref{first} we know
that $p_a(C)\leq G^*(r,d,s,\pi',-\binom{\pi_0-\pi'+1}{2})$ which
is strictly less than $G^*(r,d,s,\pi,p)$ because $\pi'<\pi$ and
$d\gg s$. Therefore we may assume that $C$ is contained in a
surface $S$ of degree $s$, with sectional genus $\pi$, and
arithmetic genus $p_a(S)\geq p$. Then by Corollary \ref{essesmall}
we know $p_a(C)\leq G^*(r,d,s,\pi,p_a(S))$ which is $\leq
G^*(r,d,s,\pi,p)$ because $p_a(S)\geq p$. This establishes the
upper bound.

\smallskip
(b) Previous argument also shows that if $p_a(C)$ reaches the
upper bound then $C$ is contained in a surface of degree $s$,
sectional genus $\pi$, and arithmetic genus $p_a(S)\geq p$. Since
$G^*(r,d,s,\pi,p_a(S)) = G^*(r,d,s,\pi,p)$ then $p_a(S)\leq p$,
hence $p_a(S)=p$.

\smallskip
(c) Taking into account Remark \ref{stima} (i) below, one may
construct a.C.M. extremal curves on the cone over $\Sigma$ exactly
as in the proof of Theorem \ref{first}. We omit the details.

\smallskip
(d) The bound is sharp when $p=-\binom{\pi_0-\pi+1}{2}$ by Theorem
\ref{first}. Next let $\Sigma'\subset \mathbb{P}^{r-1+\pi_o-\pi}$
be a smooth Castelnuovo curve of degree $r-1\leq s\leq 2r-4$. By
\cite{AR}, Theorem 2.6, p. 8, we know that a general projection
$\Sigma\subset \mathbb{P}^r$ of $\Sigma'$ remains $2$-normal. By
Proposition \ref{gl} it follows that $\Sigma$ is $k$-normal for
any $k\geq 2$. Therefore $\dim_{\mathbb
C}M(\Sigma)=h^1(\mathbb{P}^{r-1},\mathcal{I}_{\Sigma}(1))=\pi_0-\pi$.
By property (c) this proves the sharpness of the bound  in the
case $p=-(\pi_0-\pi)$. As for the case $p=0$, let $S'\subset
\mathbb{P}^{r+\pi_0-\pi}$ be a cone over a Castelnuovo curve of
degree $r-1\leq s\leq 2r-4$ as in \cite{CCD}, Example 6.4 and 6.5,
and let $C'\subset S'$ be an extremal curve with arithmetic genus
$p_a(C')=G(r+\pi_0-\pi,d,s)$. Projecting isomorphically in
$\mathbb{P}^{r}$ we get extremal curves with genus
$G^*(r,d,s,\pi,0)=G(r+\pi_0-\pi,d,s)$. Therefore the bound
$G^*(r,d,s,\pi,p)$ is sharp also when $p=0$.

This concludes the proof of Proposition \ref{second}.

\bigskip
\begin{remark}\label{stima}

\smallskip (i) Let $S\subset\mathbb{P}^r$ be an integral nondegenerate surface
of degree $r-1\leq s \leq 2r-4$, with general hyperplane section
$\Sigma$ of arithmetic genus $\pi$. Fix an integer $d\gg s$ and
consider the following numerical function
$$
h_{d,\Sigma}(i):=
\begin{cases}
1-\pi+is-h^1(\mathbb{P}^{r-1},\mathcal I_{\Sigma}(i))\quad
{\text{if $1\leq i\leq m$}}\\
d-\max\left(0,\left[\frac{2\pi-(s-1-\epsilon)}{2}\right]\right)\quad
{\text{if $i=m+1$}}\\ d\quad {\text{if $i\geq m+2$}}.
\end{cases}
$$
Observe that $h_{d,\Sigma}(i)=h_{\Sigma}(i)$ for $1\leq i\leq m$.
Using the same argument as in Lemma \ref{ons} we see that for any
curve $C\subset S$ of degree $d$ one has $h_{\Gamma}(i)\geq
h_{d,\Sigma}(i)$ for any $i$, and so
\begin{equation}\label{nbound}
p_a(C)\leq \sum_{i=1}^{+\infty}(d-h_{\Gamma}(i))\leq
\sum_{i=1}^{+\infty}(d-h_{d,\Sigma}(i))=G^*(r,d,s,\pi,-\dim_{\mathbb{C}}
M(\Sigma))
\end{equation}
where $M(\Sigma)$ denotes the Hartshorne-Rao module of $\Sigma$.
This is another "natural" upper bound for $p_a(C)$. However notice
that by Lemma \ref{arithmeticgenus} we know that
$p_a(S)=-\dim_{\mathbb{C}} M(\Sigma)+\sum_{i=1}^{+\infty}\mu_i$,
and therefore the bound appearing in Corollary \ref{essesmall} is
more fine than this new bound (\ref{nbound}), i.e.
$$
G^*(r,d,s,\pi,p_a(S))\leq G^*(r,d,s,\pi,-\dim_{\mathbb{C}}
M(\Sigma)).
$$
The inequality can be strict. For example, this is the case for a
non linearly normal smooth surface $S$ of arithmetic genus
$p_a(S)=0$. In fact for such a surface we have $M(S)\neq 0$, and
therefore $\sum_{i=1}^{+\infty}\mu_i>0$.

\smallskip
(ii) Combining the examples in \cite{Park}, p. 14, Table 1, with
Proposition \ref{second}, (c), one may construct other examples of
extremal curves with genus $G^*(r,d,s,\pi,p)$.

\smallskip (iii)  Let $X$ be a ruled surface over a smooth curve $R$ of genus
$\pi$, defined by the normalized bundle $\mathcal E=\mathcal
O_{R}\oplus \mathcal O_{R}(-\mathfrak e)$, where $\mathfrak e$ is
a fixed divisor on $R$ of degree $-e\leq -2$ \cite{Hartshorne}.
Let $\mathfrak n$ be a divisor on $R$ of degree $n\geq 2\pi+1$. By
(\cite{Hartshorne}, Ex. 2.11, pg. 385), we know that $\Sigma:=
R_0+\mathfrak n f$ is very ample on $X$ (here $R_0$ denotes a
section of $X$ with $\mathcal O_{X}(R_0)\cong\mathcal
O_{\Ps(\mathcal E)}(1)$, and $f$ a fibre of the ruling $X\to R$).
As in the proof of (\cite{Hartshorne}, Theorem 2.17, pg. 379), we
see that the complete linear system $|\,\Sigma\,|$ embeds $X$ in
$\Ps^{r+1}$ as a linearly normal surface $S$ of degree $s$,
sectional genus $\pi$ and arithmetic genus
$p_a(S)=-\pi=-(\pi_0(s,r)-\pi)$, with $s=2n-e$ and $r+1=s+1-2\pi$.
In particular $r\leq s\leq 2r-4$. Now let $C$ be any curve on $S$
of degree $d$. For a suitable integer $a$ and divisor ${\mathfrak
b}_a$ on $R$ of degree $b_a=1+\epsilon-(a+1)s$ we have $C\in
|\,(m+a+1)\Sigma+\mathfrak b_af\,|$. Taking into account that the
canonical divisor class of $S$ is
$|\,K_S\,|=|\,-2\Sigma+(s+2\pi-2)f\,|$, by the adjunction formula
we may compute the arithmetic genus of $C$, which is equal to
$$
g(a):={\binom{m}{2}}s+m(\epsilon+\pi)+\pi-\frac{s}{2}a^2+a(\pi+\epsilon-\frac{s}{2}).
$$
Taking $a=0$,  we deduce that, in the case $2\pi\leq
s+1-\epsilon$, there are smooth curves $C$ on $S$ with maximal
genus $g(0)=G^*(r+1,d,s,\pi,p)$, with $p=-\pi=-(\pi_0(s,r)-\pi)$.
Projecting isomorphically $S$ in $\mathbb{P}^r$, these examples
show the existence of smooth extremal curves with genus
$G^*(r,d,s,\pi,-(\pi_0(s,r-1)-\pi)+1)$ which are not a.C.M.. By
contrast notice that in this range (i.e. $p=-(\pi_0-\pi)+1$)
Proposition \ref{second}, (c), combined with the examples in
\cite{Park}, p. 14, Table 1, proves also the existence  of a.C.M.
extremal curves. So in certain range one can find both a.C.M. and
not a.C.M. extremal curves. Therefore the classification of
extremal curves appears somewhat complicated. Projecting in lower
dimensional subspaces, this argument works well also for other
values of $p\geq -(\pi_0-\pi)$.

\smallskip (iv)  In the case $p=0$, any
extremal curve $C$ cannot be a.C.M.. In fact if $C$ would a.C.M.
then the surface $S$ (of degree $s$, sectional genus $\pi$ and
arithmetic genus $p_a(S)=0$) on which it lies should be a.C.M. in
view of Remark \ref{ms}. This is impossible when $\pi<\pi_0$.

\smallskip
(v) Let $C\subset \mathbb{P}^{r}$ be an extremal curve in the case
$p=-(\pi_0-\pi)$, contained in a cone over a curve $\Sigma\subset
\mathbb{P}^{r-1}$ with $\dim_{\mathbb{C}}M(\Sigma)=\pi_0-\pi$.
Then we have $h_{\Gamma}(2)=h_{\Sigma}(2)=1-\pi+2s$. On the other
hand, the Hilbert function at level $2$ of the general hyperplane
section of an extremal curve  with genus $G(r,d,s+1)$ is equal to
$h_{r,d,s+1,\pi'_0}(2)=s+r+3$, which is strictly less than
$h_{\Gamma}(2)$ as soon as $\pi_0-\pi>3$. Therefore we see that
(at least in this case) there is no a minimal Hilbert function for
the general hyperplane section of a curve satisfying the
conditions in Proposition \ref{second}.

\smallskip (vi) If $S$ is smooth then $p_a(S)\geq -\pi$ and so inequality
(\ref{lessesmall}) implies  $p_a(C)\leq G(r,d,s,\pi,-\pi)$.

\smallskip (vii) From the proof of Corollary \ref{essesmall} we see  that the bound
$$
p_a(C)\leq {\binom{m}{2}}s+m(\epsilon+\pi)-p_a(S)
$$
holds true for any $s$ and $d\gg s$, if $2\pi\leq s+1-\epsilon$.
So when $\pi=0$ then we have the bound
$$
p_a(C)\leq {\binom{m}{2}}s+m\epsilon-p_a(S).
$$
In certain cases it is sharp. In fact, let $S\subset \Ps^4$ be a
general projection of a smooth rational normal scroll
$S'\subset\Ps^{s+1}$, and let $\delta_S$ be the number of double
points of $S$. From the double point formula we know that
$\delta_S={\binom{s-2}{2}}$. On the other hand we have
$p_a(S)=-\delta_S$. So previous bound becomes
$$
p_a(C)\leq {\binom{m}{2}}s+m\epsilon+{\binom{s-2}{2}}.
$$
Now take a Castelnuovo's curve $C'\subset S'$ of degree $d\gg s$
passing through the double point set of $S'$. Then the projection
$C$ of $C'$ acquires $\delta_S$ nodes and so
$$
p_a(C)=p_a(C')+\delta_S=
{\binom{m}{2}}s+m\epsilon+{\binom{s-2}{2}}.
$$

\smallskip (viii) The arithmetic genus of a curve $C$ complete
intersection of a surface $S$  with a hypersurface of degree $m+1$
is $p_a(C)={\binom{m}{2}}s+m(\epsilon+\pi)+\pi$, where $s$ and
$\pi$ are the degree and sectional genus of $S$. On the other
hand, in this range, i.e. when $\epsilon=s-1$, we have
$G^*(r,d,s,\pi,p)={\binom{m}{2}}s+m(\epsilon+\pi)-p+\pi$, which is
strictly greater  than $p_a(C)$ when $p<0$. In other words, in
contrast with the classical case, in our setting complete
intersections are not extremal curves.

\smallskip (ix) Let $C$ be an extremal curve as in Theorem \ref{first},
and assume $\epsilon =s-1$. Let $S$ be the surface of degree $s$,
sectional genus $\pi$ and arithmetic genus
$p_a(S)=-{\binom{\pi_0-\pi+1}{2}}$ on which $C$ lies. We remark
that $S$ cannot be locally Cohen-Macaulay.  In fact, by the proof
of Lemma \ref{ons} we see that since $C$ is extremal then $\Gamma$
is the complete intersection of $\Sigma$ with a hypersurface of
degree $m+1$. Since $C$ is a. C. M. one may lift such a
hypersurface to a hypersuface $F\subset \Ps^{r}$ of degree $m+1$
containing $C$ and not containing $S$. If $S$ would be locally
Cohen-Macaulay then $C$, as a scheme, would be the complete
intersection of $S$ with $F$ for degree reasons. This is absurd in
view of previous remark (viii).
\end{remark}

{\bf{Aknowledgements}}

We would like to thank Ciro Ciliberto for valuable discussions and
suggestions on the subject of this paper.

\end{document}